\documentclass[11pt,a4paper]{article}
\usepackage[utf8]{inputenc}
\usepackage[english]{babel}
\usepackage{geometry}
\usepackage{amsmath}
\usepackage{amssymb}
\usepackage{mathtools}
\usepackage{booktabs}
\usepackage{hyperref}
\usepackage{amsthm}
\usepackage{enumitem}
\usepackage{tikz}
\usepackage{nccmath}
\usepackage{csquotes}

\usepackage{esint}
\usepackage{tikz-cd}
\usepackage{tikz}
\usepackage{bbm}

\usepackage{dsfont}
\usepackage{hyperref}

\usepackage{mathrsfs}
\usepackage{braket}
\usepackage{mathtools}
\usepackage{booktabs}
\usepackage{caption}
\usepackage{cancel}
\usepackage{subfig}
\usepackage{epstopdf}

\mathtoolsset{showonlyrefs}

\numberwithin{equation}{section}
\newcommand{\numberset}{\mathbb}

\newcommand{\R}{\numberset{R}}

\newcommand{\dz}{d_0}
\newcommand{\nz}{\nu_0}
\newcommand{\dhz}{\hat{d}_0}

\newcommand{\uu}{\hat{u}}
\newtheorem{lemma}{Lemma}[section]
\newtheorem{theorem}{Theorem}
\newtheorem{prp}[lemma]{Proposition}
\newtheorem{cor}[lemma]{Corollary}
\newtheorem{defn}[lemma]{Definition}

\newtheorem{remark}[lemma]{\textbf{Remark}}
\usepackage[style=alphabetic]{biblatex}
\renewbibmacro{in:}{%
  \ifentrytype{article}
    {}
    {\bibstring{in}%
     \printunit{\intitlepunct}}}
\title{Radial blow-up standing solutions for the semilinear wave equation}
\author{Maissâ Boughrara and Hatem Zaag\\
\textit{\small Université Sorbonne Paris Nord, LAGA, CNRS (UMR 7539), F-93430 Villetaneuse, France.}
}
\date{}

\begin{document}
\maketitle

\begin{abstract}
We consider the semilinear wave equation with a power nonlinearity in the radial case. Given $r_0>0$, we construct a blow-up solution such that the solution near $(r_0,T(r_0))$ converges exponentially to a soliton. Moreover, we show that $r_0$ is a non-characteristic point. For that, we translate the question in self-similar variables and use a modulation technique. We will also use energy estimates from the one dimensional case treated in \cite{MZ2007}. Of course because of the radial setting, we have an additional gradient term which is delicate to handle. That's precisely the purpose of our paper.
\end{abstract}

%\tableofcontents
\section{Introduction}
We are concerned with  blow-up phenomena for real-valued solutions of the following semilinear wave equation:

\begin{equation}\label{eq initial}
\left\lbrace
\begin{aligned}
&\partial^2_t U=\Delta U+|U|^{p-1}U,\\
&U(.,0)=U_0,\\
&U_t(.,0)=U_1,
\end{aligned}
\right.
\end{equation}
where $1<p$, if $N\geq 2$ then $p\leq 1+\frac{4}{N-1}$ , and $(U_0,U_1)\in H^1_{loc,u}(\R^N) \times L^2_{loc,u}(\R^N)$. The space $L^2_{loc,u}$ is the set of all $v$ in $L^2_{loc}$ such that

\begin{equation}\label{norm L^2_{loc,u}}
||v||_{L^2_{loc,u}}:=\underset{a\in \R^N}{\sup}\left(\int_{|x-a|<1}|v(x)|^2  d x\right)^{\frac{1}{2}}<+\infty,
\end{equation}
and the space $H^1_{loc,u}= \{ v\in L^2_{loc,u}|\nabla v\in L^2_{loc,u} \}$. The Cauchy problem of \eqref{eq initial} can be solved locally in time in the space $H^1_{loc,u} \times L^2_{loc,u}$ (see \cite{GSV1992} and \cite{LS1995}). For the existence of blow-up solutions, it follows from ODE techniques or the energy-based blow-up criterion of Levine in \cite{L1974}.

Concerning the blow-up behavior for general solutions (i.e. with no symmetry assumptions), we first mention the result of Merle and Zaag in \cite{MZ2003}, \cite{MZ2005MA} and \cite{MZ2005} who showed that the blow-up rate is given by $C(T-t)^{-\frac{2}{p-1}}$, i.e. the solution of the associated ODE $u"=u^p$, where $T$ is the (local) blow-up time.

As for the derivation of the asymptotic behavior of the blow-up, it depends on the dimension. Indeed, if $N=1$, all has been described by Merle and Zaag in a series of papers (\cite{MZ2007}, \cite{MZ2008}, \cite{MZ2012}, \cite{MZ2012E}...etc). In higher dimension $N\geq 2$, some partial results are available in \cite{MZ2015CMP}, \cite{MZ2016}... etc.

The main obstruction in the generalization of the one-dimensional case to higher dimension stems from the lack of a classification of all self-similar solutions of \eqref{eq initial} in the energy spaces. For that reason, Merle and Zaag addressed in \cite{MZ2011BSM} the radial case, which reduces to a perturbed one dimensional case, if we focus on the blow-up behavior outside the origin. In particular, they showed that the solution in similarity variables converges to the one-dimensional soliton. However, they gave no example of initial data which behaves like that. This is precisely our aim in this paper. Therefore, we assume here that $U$ is a radial solution. Introducing

\begin{equation}\label{radial solution}
    u(r,t)=U(x,t) \text{ where } r=|x|,
\end{equation}
we rewrite \eqref{eq initial} as
\begin{equation}\label{eq radial}
\left\lbrace
\begin{aligned}
&\partial^2_t u=\partial^2_r u+\frac{(N-1)}{r}\partial_r u+|u|^{p-1}u,\\
&\partial_r u(0,.)=0,\\
&u(.,0)=u_0,\\
&u_t(.,0)=u_1.
\end{aligned}
\right.
\end{equation}
If $u$ blows up then we define a 1-Lipschitz curve $\Gamma=\{(r,T(r)), r\geq 0\}$ such that the domain of definition $D$ called the maximum influence domain of $u$ is written as the following 
$$D=\{(r,t)|t<T(r)\}.$$
$\Gamma$ is called the blow-up graph of $u$.  A point $r_0$ is a non-characteristic point if there are
\begin{equation}\label{def non charac}
\delta_0\in (0,1) \text{ and  }t_0 <T(r_0) \text{ such that } u \text{ is defined on } \mathcal{C}_{r_0,T(r_0),\delta_0} \cap \{t \geq t_0\}\cap \{r \geq 0\}, 
\end{equation}
where $C_{\bar r,\bar t,\bar \delta} =\{(r,t) | t<\bar t - \bar \delta|r - \bar r|\}$. We denote by $\mathcal{R} \subset \R^+$ (resp. $\mathcal{S} \subset \R^+$) the set of non characteristic (resp. characteristic) points.

\subsection{Explicit solutions}
Given $r_0>0$, $T>0$, we define for all $|r-r_0|<T-t$, $|d| < 1$ and $\nu>-1+|d|$ the following family of functions
\begin{equation}\label{def uu}
    \uu(d,\nu,r_0,T,r,t)=\kappa_0\frac{(1-d^2)^{\frac{1}{p-1}}}{((1+\nu)(T-t)+d(r-r_0))^{\frac{2}{p-1}}}\text{ where } \kappa_0=\left(\frac{2(p+1)}{(p-1)^2}\right)^{\frac{1}{p-1}}.
\end{equation}
One may see that $\uu(d,0,r_0,T)$ is a blow-up solution of \eqref{eq radial} in one dimension, for $\nu=0$. Given $r_0>0$ and $T_0\leq T(r_0)$, we introduce the following similarity variables
\begin{equation}\label{similarity variables}
\begin{split}
y=\frac{r-r_0}{T_0-t},\ s=-\log(T_0-t).\\
w_{r_0,T_0}(y,s)=(T_0-t)^{\frac{2}{p-1}}u(r,t).
\end{split}
\end{equation}
The function $w=w_{r_0,T_0}$ satisfies the following equation for $|y|<1$ and $s>-\log T_0$:
\begin{equation}\label{eq similary}
\partial^2_s w=\mathcal{L}w - \frac{2(p+1)}{(p-1)^2}w+|w|^{p-1}w-\frac{p+3}{p-1}\partial_s w-2y\partial^2_{y,s} w+e^{-s}\frac{(N-1)}{r_0+ye^{-s}}\partial_y w,
\end{equation}
where 
\begin{equation}\label{def L}
\mathcal{L}w=\frac{1}{\rho}\partial_y(\rho(1-y^2)\partial_yw) \text{ and } \rho(y)=(1-y^2)^{\frac{2}{p-1}}.
\end{equation}
%Let us introduce the following functional of \eqref{eq similary}
%\begin{equation*}
%E(w)=\int^1_{-1}(\frac{1}{2}(\partial_s w)^2+\frac{1}{2}(\partial_y w)^2(1-y^2)+\frac{(p+1)}{(p-1)}w^2-\frac{1}{p+1}|w|^{p+1})\rho dy. 
%\end{equation*}
%We define the following norm
%\begin{equation}\label{def norm H*}
%||(q_1,q_2)||_{\mathcal{H}^*}^2=\int^{r_0+(T(r_0)-t)}_{r_0-(T(r_0)-t)}\left(q_1^2+(q_1')^2((T(r_0)-t)^2-(r-r_0)^2)+q_2^2\right)\mu(r) dr,
%\end{equation}
%with $\mu(r)=((T(r_0)-t)^2-(r-r-0)^2)^{\frac{2}{p-1}}$.
We have for all $|d| < 1$ the following solitons defined by
\begin{equation}\label{def kappa}
\kappa(d,y)=\kappa_0 \left(\frac{1-d^2}{(1+dy)^2}\right)^{\frac{1}{p-1}} \text{ where } |y|<1,
\end{equation}
and $\kappa_0$ is given in \eqref{def uu}. Note that $\kappa(d)$ is a stationary solution of \eqref{eq similary} in one space dimension. We introduce the "generalized solitons" or solitons for short defined for all $|y|<1$, $|d| < 1$ and $\nu>-1+|d|$, by $\kappa^*(d, \nu, y)=\left(\kappa_1^*, \kappa_2^*\right)(d, \nu, y)$ where
\begin{align}
& \kappa_1^*(d, \nu, y)=\kappa_0 \frac{\left(1-d^2\right)^{\frac{1}{p-1}}}{(1+d y+\nu)^{\frac{2}{p-1}}}, \label{kappa*1}\\
& \kappa_2^*(d, \nu, y)=\nu \partial_\nu \kappa_1^*(d, \nu, y)=-\frac{2 \kappa_0 \nu}{p-1} \frac{\left(1-d^2\right)^{\frac{1}{p-1}}}{(1+d y+\nu)^{\frac{p+1}{p-1}}}.\label{kappa*2}
\end{align}

Notice also that for all $\mu\in \R$, $\overline{\kappa^*}(d,\mu,y,s)=\kappa^*(d,\mu e^{ s},y)$ is a solution of the first order vector-valued formulation of \eqref{eq similary} in dimension one.

\subsection{Main results}
We first introduce the following spaces:
\begin{equation}\label{def H}
\mathcal{H}=\{(q_1,q_2)\in H^1_{loc} \times L^2_{loc}|\ ||(q_1,q_2)||_{\mathcal{H}}^2=\int^1_{-1}\left(q_1^2+(q_1')^2(1-y^2)+q_2^2\right)\rho dy<+\infty\},
\end{equation}
\begin{equation}\label{def H0}
\mathcal{H}_0=\{q_1\in H^1_{loc} |\ ||q_1||_{\mathcal{H}_0}^2=\int^1_{-1}\left(q_1^2+(q_1')^2(1-y^2)\right)\rho dy<+\infty\},
\end{equation}
where $\rho$ is given in \eqref{def L}. Note that the norm of the $\mathcal{H}$ space can be defined by the inner product 
\begin{equation}\label{norm H}
    ||q||_{\mathcal{H}}=\sqrt{\phi(q,q)},
\end{equation}
where 
\begin{equation}\label{def phi}
\phi(q,r)=\int^1_{-1}\left(q_1r_1+q_1'r_1'(1-y^2)+q_2r_2\right)\rho dy.
\end{equation}
Using integration by parts and the definition of $\mathcal{L}$ in \eqref{def L}, we have the following identity
\begin{equation}
\phi(q,r)=\int^1_{-1}\left(q_1(-\mathcal{L}r_1+r_1)+q_2r_2\right)\rho dy.
\end{equation}
The following is our main result giving us the existence of blow-up solutions converging to solitons, for well-chosen initial data.
\begin{theorem}[A solution converging to a soliton in similarity variables]\label{Thm}
    There exists $\delta\in (0,1)$, such that for all $r_0>0$ and $\dhz\in (-1,1)$, there exist  $s_0> -\log r_0$, $\nz, \dz \in \R $, and $\varepsilon_0> 0$ such that equation \eqref{eq radial} with the initial data given by 
    
    \begin{equation}\label{initial data u}
    \binom{u_0(r)}{u_1(r)}=\chi(r)\binom{\uu(\dz,\nz,r_0,T_0,r,0)}{\frac{1}{1+\nz}\partial_t\uu(\dz,\nz,r_0,T_0,r,0)}, \text{ for all } r\geq r_0,
    \end{equation}
    with $T_0=e^{-s_0}$, has a solution $u(r,t)$ defined for all $(r,t)\in (\R^+)^2$ such that $0\leq t <T(r)$ for some 1-Lipschitz curve $r\mapsto T(r)$. Moreover, $r_0$ is a non-characteristic point, $T(r_0)=T_0$ and for $w=w_{r_0,T(r_0)}$, we have $(w(s),\partial_s w(s))\in \mathcal{H}$ satisfies

   \begin{equation}\label{conv in H}
\left|\left|\binom{w(s)}{\partial_s w(s)}-\binom{\kappa(\dhz)}{0}\right|\right|_{\mathcal{H}}\leq Ce^{-\delta(s-s_0)-s_0}.    
\end{equation}
    %\begin{equation}\label{conv sol}
    %\begin{split}
   % \left|\left|\binom{u(t)}{\partial_t u(r,t)(T_0-t)-\partial_ru(r-r_0)-\frac{2}{p-1}u}\right.\right.&\left.\left.-\binom{(T_0-t)^{-\frac{2}{p-1}}\kappa\left(\dhz,\frac{r-r_0}{T_0-t}\right)}{0}\right|\right|_{\mathcal{H}^*}\\
   % &\leq CT_0^{1-\delta}(T_0-t)^{\delta+\frac{1}{2}},
   % \end{split}
   % \end{equation}
    where $\uu$ and $\kappa$ are given in \eqref{def uu}, \eqref{def kappa}, and $\chi\in \mathcal{C}^\infty_0$, with $\chi\equiv 1$ on $[r_0-T_0-\varepsilon_0,r_0+T_0+\varepsilon_0]$, $\chi\equiv 0$ on $\R\backslash[r_0-T_0-2\varepsilon_0,r_0+T_0+2\varepsilon_0]$.
    
\end{theorem}
\begin{remark}
    For the case of $N=1$ or of $\dhz=0$, we have an explicit solution. Indeed:
    \begin{itemize}
        \item if $\dhz=0$, then $u(r,t)=\kappa(T_0-t)^{-\frac{2}{p-1}}$,
        \item  if $N=1$, then $u(r,t)=\uu (\dhz,0,T_0,r,t)$ defined in \eqref{def uu},
    \end{itemize}
    for all for $T_0>0$. Therefore, the result is meaningful when $N\geq 2$ and $\dhz\neq0$.
\end{remark}
\begin{remark}

Our proof relies on the use of similary variables \eqref{similarity variables}. In that setting, we will linearize equation \eqref{eq similary} around the soliton $\kappa(\dhz)$, and use a modulation technique together with energy methods from the one dimensional case treated in \cite{MZ2007}. Accordingly our focus will be on the term $e^{-s}\frac{(N-1)}{r_0+ye^{-s}}\partial_y w$ which needs a delicate treatment that is precisely the goal of our paper.
\end{remark}

Using the converging technique of \cite{MZ2005} and proceeding as in \cite{MZ2007} for $N=1$ and \cite{MZ2011BSM} for $N\geq 2$ in the radial case, we derive the following stronger convergence:
%From Corollary 4 in \cite{MZ2007} and the uniqueness converge of \eqref{conv in H} we have the following corollary: 
\begin{cor}[Convergence in $H^1\times L^2(-1,1)$]\label{cor H1L2 conv}
Under the hypotheses of Theorem \ref{Thm}, it holds that:
\begin{equation}\label{conv in L2*H1}
\left|\left|\binom{w(s)}{\partial_s w(s)}-\binom{\kappa(\dhz)}{0}\right|\right|_{H^1(-1,1)\times L^2(-1,1)}\underset{s\rightarrow +\infty}{\longrightarrow}    0.
\end{equation}
\end{cor}

We can restate this corollary in the initial variables, but let us first give the following definition of the equivalence of functions in the normalized $L^2_{\{|r-r_0|<T_0-t\}}$.
\begin{defn}[Equivalence of functions in the normalized $L^2_{\{|r-r_0|<T_0-t\}}$ norm]
    Let $g=g(r,t)$ and $f=f(r,t)$ be two functions defined on $(\R^+)^2$, such that $|r-r_0|<T_0-t$. We say that $f$ and $g$ are equivalent ($f\sim g$) if
    $$\fint_{|r-r_0|<T_0-t} (f(r,t)-g(r,t))^2dr=o\left(\fint_{|r-r_0|<T_0-t} f(r,t)^2dr\right)\text{ as }t\mapsto T.$$
\end{defn}
Therefore, using \eqref{similarity variables}, we obtain the following equivalences in the initial set of variables $u(r,t)$:

\begin{cor}[Profile in the $u(x,t)$ variables]\label{equiv initial variable}
Under the assumptions of Theorem \ref{Thm}, we have the following equivalences as $t\rightarrow T_0$:
\begin{equation}
\begin{split}
& u\sim \uu(\dhz,0,r_0,T_0),\ \partial_r u\sim \partial_r \uu(\dhz,0,r_0,T_0),\text{ and } \partial_tu\sim \partial_t\uu(\dhz,0,r_0,T_0).
\end{split}
\end{equation}
\end{cor}

From the classification of the blow-up behavior outside the origin for radial solutions given in \cite{MZ2011BSM}, we have the following corollary:
\begin{cor}[Differentiability and stability with repect to the blow-up point]\label{cor stability 1}
Under the hypotheses of Theorem \ref{Thm}, the following holds:
    \begin{itemize}
        \item \textbf{(Differentiability of $r\mapsto T(r)$)} There exists $\varepsilon>0$, such that the curve $T$ is of class $C^1$ in $[r_0-\varepsilon,r_0+\varepsilon]$. Moreover, $\dhz=T'(r_0)$.
        \item \textbf{(Blow-up behavior near $r_0$)} For all $r_1\in [r_0-\varepsilon,r_0+\varepsilon]$, we have that
        \begin{equation}
        \left|\left|\binom{w_{r_1,T(r_1)}(s)}{\partial_s w_{r_1,T(r_1)}(s)}-\binom{\kappa(T'(r_1))}{0}\right|\right|_{H^1(-1,1)\times L^2(-1,1)}\underset{s\rightarrow +\infty}{\longrightarrow}    0.
        \end{equation}
    \end{itemize}
\end{cor}

Merle and Zaag in \cite{MZ2011BSM} proved the stability of the profile \eqref{conv in H} in the radial setting with respect to the initial data. This gives us the stability of the solution constructed in Theorem \ref{Thm}. indeed:

%With a careful reading of the proof of Theorem 1 in \cite{MZ2016} and the proof of Theorem 2 in \cite{MZ2015CMP}, one may see that the two results hold also for radial symmetry case. Therefore, combining the two results, we obtain the following corollary:

%\begin{cor}[Stability in the radial setting]
   % Under the hypotheses of Theorem \ref{Thm}, there exist $\varepsilon_0, \bar \varepsilon>0$, $K_0,\mu_0>0$ such that for all $(u_0,u_1)$ such that $||(u_0,u_1)-(\hat u_0,\hat u_1)||_{H^1\times L^2}\leq \varepsilon_0$, the solution
% $u(r,t)$ of \eqref{eq radial} with initial data $(u_0,u_1)$ blows up in finite time with a blow-up curve $r\rightarrow \bar T(r)$, and 
% \begin{equation}
        %\left|\left|\binom{w_{r_0,\bar T(r_0)}(s)}{\partial_s w_{r_0,\bar T(r_0)}(s)}-\binom{\kappa(\bar T'(r_0))}{0}\right|\right|_{H^1(-1,1)\times L^2(-1,1)}\leq K_0\bar \varepsilon e^{-\mu_0 s},
        %\end{equation}
%and 
%$$| \arg \tanh|T'(r_0)|-\arg \tanh|\bar T'(r_0)|| + \frac{|T'(r_0)-\bar %T'(r_0)|}{\sqrt{1+|\bar T'(r_0)|}}\leq  K_0\bar \varepsilon .$$
%\end{cor}

\begin{cor}[Stability with respect to the initial data in the radial setting]\label{cor stability 2}

   Let us denote by $\bar u$, the solution constructed in Theorem \ref{Thm} and a bar on all the items associated with $\bar u$. Under the hypotheses of Theorem \ref{Thm}, there exist $ \eta>0$, $K_0>0$ such that for all $(U_0,U_1)$ radial initial data such that $||(U_0,U_1)-(\bar U_0,\bar U_1)||_{H^1\times L^2(\R^N)}\leq \varepsilon_0$, the solution
 $u(r,t)$ of \eqref{eq radial} with initial data $(u_0,u_1)$ blows up in finite time with a blow-up curve $r\rightarrow  T(r)$, and 
 \begin{equation}
        \left|\left|\binom{w_{r_0, T(r_0)}(s)}{\partial_s w_{r_0, T(r_0)}(s)}-\binom{\kappa( T'(r_0))}{0}\right|\right|_{H^1\times L^2(-1,1)}\underset{s\rightarrow +\infty}{\longrightarrow}    0
        \end{equation}
and 
$$| \arg \tanh|T'(r_0)|-\arg \tanh|\bar T'(r_0)|| + \frac{|T'(r_0)-\bar T'(r_0)|}{\sqrt{1+|\bar T'(r_0)|}}\leq  K_0\eta .$$
\end{cor}

\begin{remark}
    Here $U_0(x)=u_0(|x|)$ as in \eqref{radial solution} and the same for $U,\bar U_0$ and $\bar U$.
\end{remark}
\begin{remark}
    In fact, Merle and Zaag in \cite{MZ2015CMP} proved the stability in the general case, even in the non-radial setting. This implies that the constructed solution is stable with respect to non-radial perturbation in initial data. In order to avoid introduction of more notations, we invite the reader to guess the non-radial case of corollary \ref{cor stability 2}, simply by checking \cite{MZ2015CMP}.
\end{remark}

This paper is organized as the following:\\
\textbf{Section \ref{setting prob}:} We reformulate the goal of this work in similarity variables and introduce some useful tools for the proof.\\
\textbf{Section \ref{section dyn in sim var}:} We obtain an exponential bound of the difference between the solution of the vector-valued version of \eqref{eq similary} and the generalized soliton thanks to some modulation technique and energy estimates.\\
\textbf{Section \ref{section shrink set}:} We prove estimate \eqref{conv in H} by trapping the parameter $d, \nu$ and the difference between the solution and the generalized soliton in some shrinking set.\\
\textbf{Section \ref{proof of thm}:} We conclude the proof of Theorem \ref{Thm} by obtaining the initial data \eqref{initial data u} in the initial variables. Note that we don't give the proof of corollaries \ref{cor H1L2 conv}, \ref{equiv initial variable}, \ref{cor stability 1} and \ref{cor stability 2}, as they follow in a straightforward way from earlier litarature.

\textbf{Acknowledgement:} Hatem Zaag wishes to thank Pierre Raphaël and the ”SWAT” ERC project for their support. Moreover, the authors would like to express their gratitude to the referee for the effort they dedicated to reviewing the article and providing valuable suggestions which contributed to the improvement of this paper.
\section{Setting of the problem}\label{setting prob}
Given $r_0>0$ and $\dhz\in (-1,1)$, we aim at constructing a solution $w$ to the Cauchy problem of equation \eqref{eq similary}, such that 
\begin{equation}\label{prp conv}
\left|\left|\binom{w(s)}{\partial_s w(s)}-\binom{\kappa(\dhz)}{0}\right|\right|_{\mathcal{H}}\longrightarrow 0 \text{ as } s \rightarrow +\infty.
\end{equation}
For that aim, we will take initial data

\begin{equation}\label{initial data soliton}
\binom{w(y,s_0)}{\partial_s w(y,s_0)}=\kappa^*(\dz, \nz ,y),
\end{equation}
where $s_0> -\log r_0$ will be fixed large enough later, $|\dhz-\dz|\leq A e^{-s_0}$, and $|\nz|\leq A e^{-s_0}$, and the constant $A>0$ will be fixed later. The aim of Sections \ref{section dyn in sim var} and \ref{section shrink set} is to prove the following:

\begin{prp}[Solution converging to a soliton in the $w(y,s)$ setting]\label{sol conv solition}
There exists $\delta\in (0,1)$, for all $r_0>0$ and $\dhz\in (-1,1)$, there exist $s_0> -\log r_0$, $\nz, \dz \in \R $ such that equation \eqref{eq similary} with the initial data given by \eqref{initial data soliton} is defined for all $(y,s)\in (-1,1)\times [s_0,\infty)$ and satisfies $(w(s),\partial_s w(s))\in \mathcal{H}$, and there exists $C>0$ such that
\begin{equation}\label{profile similary variables}
\left|\left|\binom{w(s)}{\partial_s w(s)}-\binom{\kappa(\dhz)}{0}\right|\right|_{\mathcal{H}}\leq Ce^{-\delta(s-s_0)-s_0}.    
\end{equation}

\end{prp}
A solution of equation \eqref{eq similary} with initial data \eqref{initial data soliton} will be denoted by $w(s_0,\dz,\nz,y,s)$, or, when there is no ambiguity, by $w(y,s)$ or $w(s)$ for short. We define the vector-valued function $q=(q_1,q_2)$ as the following

\begin{equation}\label{q}
\binom{w(y,s)}{\partial_s w(y,s)}=\kappa^*(d(s),\nu(s),y)+q(y,s),
\end{equation}
where $d$ and $\nu$ will be chosen later through a modulation technique. 

We define as in \cite[Lemma 4.4]{MZ2007} and \cite[Lemma 2.4]{MZ2016}, for all $|d| < 1$ and $\lambda\in \{0, 1\}$,  $W_\lambda^d \in \mathcal{H}$ continuous in terms of $d$ by
\begin{equation}\label{def W_lambda,2}
W_{1,2}^d(y)=c_1(d) \frac{1-y^2}{(1+d y)^{\frac{2}{p-1}+1}}, \quad W_{0,2}^d(y)=c_0(d) \frac{y+d}{(1+d y)^{\frac{2}{p-1}+1}}.    
\end{equation}
$W_{\lambda, 1}^d$ is uniquely determined by the solution $r$ of the following equation
$$
-\mathcal{L} r+r=\left(\lambda-\frac{p+3}{p-1}\right) r_2-2 y r_2^{\prime}+\frac{8}{p-1} \frac{r_2}{1-y^2},
$$
with $r_2=W_{\lambda, 2}^d$ and the $C^1$ function $c_\lambda(d)>0$ are given by the following
$$c_\lambda(d)=c_\lambda(1-d^2)^{\frac{1}{p-1}},\text{ with } \frac{1}{c_\lambda}=2(\lambda+\frac{2 }{p-1})\int_{-1}^1 \left(\frac{y^2}{1-y^2}\right)^{1-\lambda}\rho(y) dy.$$
Note that $W_\lambda^d$ are eigenfunctions of $L^*_{d,0}$, the adjoint operator of $L_{d,0}$ given below in \eqref{def L_d,nu}. This operator is important for the dynamics of the equation satisfied by $q$ (see \eqref{eq q}). Since $\lambda\geq 0$, the control of the projection of $q$ on the eigenfunctions of $L_{d,\nu}$ is delicate. As a matter of fact, the eignefunctions $W_d^\lambda$ are useful to define those projections as follows: For $\lambda=0,1$ and for any $v\in \mathcal{H}$.
\begin{equation}\label{def pi}
    \pi^d_\lambda(v)=\phi(W^d_\lambda,v),
\end{equation}
where $\phi$ is defined \eqref{def phi}.

%\begin{remark}
%    If the reader is interested in knowing why we use the projection defined by $\phi$ on the eigenspaces associated to $\lambda=0,1$, we give detailed explanations in Appendices \ref{Non-negative directions of L^d} and \ref{The conjugate operator L^*_d}
%\end{remark}

\section{Dynamics in similarity variables}\label{section dyn in sim var}

The aim of this section is to obtain an exponential bound of $q$. This will be achieved in two steps, given in two separate subsections:

We first kill nonnegative modes of the linear operator, given in \eqref{Ld* and V}, around the soliton. To do such a thing, we need to choose the correct parameters in \eqref{q}, using a modulation technique.

Then, we project the equation satisfied by $q$ on the nonnegative directions and use energy estimates to derive the exponential bound on $q$.
\subsection{Modulation theory}
Here, we impose some orthogonality conditions related to nonnegative direction of the linearized operator of equation \eqref{eq q} satisfied by $q$. The following modulation technique from Merle and Zaag in \cite{MZ2012} allows us to do such a thing.
\begin{prp}[Modulation technique]\label{modulation technique}
For all $B\geq 1$, there exists $\varepsilon_0=\varepsilon_0(B)>0$, such that for all $\varepsilon\leq \varepsilon_0$, if $v\in \mathcal{H}$ and for $(\hat{d},\hat{\nu})\in (-1,1)\times \R$, such that
$$-1+\frac{1}{B}\leq \frac{\hat{\nu}}{1-|\hat{d}|}\leq B,\ ||\hat{q}||_{\mathcal{H}}\leq \varepsilon,$$
where
$$\hat{q}=v-\kappa^*(\hat{d},\hat{\nu})\ \text{ and } \ \hat{d^*}=\frac{\hat{d}}{1+\hat{\nu}}=-\tanh \hat{\xi^*}.$$
Then, there exists $(d,\nu)\in (-1,1)\times \R$ such that
\begin{enumerate}[label=(\roman*)]
    \item for $l=0, 1$, $\pi^{d^*}_l(q)=0$, where $q=v-\kappa^*(d,\nu)$.
    \item $\left| \frac{\nu}{1-|d|}-\frac{\hat{\nu}}{1-|\hat{d}|}\right|+|\xi^*-\hat{\xi^*}|\leq C(B) ||\hat{q}||_{\mathcal{H}}\leq C(B)\varepsilon$.
    \item $-1+\frac{1}{2B}\leq \frac{\nu}{1-|d|}\leq B+1$ and $||q||_{\mathcal{H}}\leq \varepsilon$.
\end{enumerate}
where $d^*=\frac{d}{1+\nu}=-\tanh\ \xi^*$.
\end{prp}
\begin{proof}
    In \cite[Proposition  3.1]{MZ2012}, the proof is given for functions near multisolitons. A careful reading allows to see that our present case, with only one soliton follows from the same proof.
\end{proof}
Taking $B=2$ and $s_0$ large enough, we can apply Proposition \ref{modulation technique} to $v=\binom{w(s_0)}{\partial_s w(s_0)}$, $\hat{d}=\dz$ and $\hat{\nu}=\nz$, introduced in \eqref{initial data soliton}. Therefore, $\hat q=0$. From the continuity of the flow associated with equation \eqref{eq similary} in $\mathcal{H}$, we have a maximal $\bar s=\bar s(s_0, \dz,\nz)$ such that the solution $w$ exists for all times $s\in [s_0,\bar s)$ and can be modulated in the sense that there exist $d=d(s)$ and $\nu=\nu(s)$, such that for all $s\in [s_0,\bar s)$
\begin{equation}\label{def w,kappa,d,nu,q}
\begin{split}
\binom{w(s,y)}{\partial_s w(s,y)}=\kappa^*(d(s),\nu(s),y)+q(s,y) \text{ with } \pi^{d^*(s)}_l(q(s))=0 \text{ for } l=0,1,
\end{split}    
\end{equation}
where  
\begin{equation}\label{def d*}
    d^*=\frac{d}{1+\nu},
\end{equation}
and
\begin{equation}\label{consequances modulation technique}
\frac{|\nu|}{1-|d|}\leq s_0^{-\frac{1}{4}},\ ||q(s)||_{\mathcal{H}}\leq s_0^{-1/2}.
\end{equation}
which lead us to the following
\begin{equation}\label{equiv d and d*}
    |d^*|<1 \text{ and }\frac{1}{2}\leq \frac{1-d^{*2}}{1-d^2}\leq 2, \text{ provided that $s_0$ is large enough.}
\end{equation}
In particular,
\begin{equation}\label{nu,d(s0)=nu0,d_0}
    (\nu(s_0),d(s_0))=(\nz,\dz).
\end{equation}
Then, two cases arise:
\begin{itemize}
    \item Either $\bar s=+\infty$.
    \item Or, $\bar s<+\infty$ and one of the symbols in \eqref{consequances modulation technique} is an $=$ at $\bar s$.
\end{itemize}
We see that controlling the solution $w\in \mathcal{H}$ is equivalent to controlling $q\in \mathcal{H}, d(s)\in (-1,1)$ and $\nu(s) \in \R$.
\subsection{Energy estimates}

Here we derive an exponential bound on $q$, using energy estimates.
\begin{prp}[An exponential bound]\label{Dynamics of the parameters}
There exist $\delta\in (0,1)$ and $C^*>0$ such that for $s_0$ large enough and for all $s\in [s_0,\bar s)$, we have
$$||q(s)||^2_{\mathcal{H}}\leq C^*e^{-\delta(s- s_0)-s_0}.$$
\end{prp}
Before starting the proof of Proposition \ref{Dynamics of the parameters}, let us first derive an equation on $q$. Using the remark below \eqref{kappa*2} and the first order vector-valued formulation of \eqref{eq similary} in dimension one, we have
\begin{equation}\label{eq kappa * bar}
\nu\partial_\nu \kappa^*=\partial_s \overline{\kappa^*}=\binom{\kappa^*_2}{\mathcal{L}\kappa^*_1 - \frac{2(p+1)}{(p-1)^2}\kappa^*_1+\kappa^{*p}_1-\frac{p+3}{p-1}\kappa^*_2-2y\partial_{y}\kappa^*_2}.
\end{equation}
From \eqref{eq similary}, \eqref{q} and \eqref{eq kappa * bar}, the equation of $q$ is given by the following 
\begin{equation}\label{eq q}
\begin{split}
\partial_s \binom{q_1}{q_2}=&L_{d,\nu}\binom{q_1}{q_2}+\binom{0}{f_{d,\nu}(q_1)}-d'\partial_d \kappa^*(d,\nu)-(\nu'-\nu)\partial_\nu \kappa^*(d,\nu)\\
&+\binom{0}{e^{-s}\frac{(N-1)}{r_0+ye^{-s}}\partial_y (\kappa^*_1(d,\nu)+q_1)},
\end{split}
\end{equation}
where
\begin{equation}\label{def L_d,nu}
L_{d,\nu}\binom{q_1}{q_2}=\binom{q_2}{\mathcal{L}q_1+\psi_{d,\nu} q_1-\frac{p+3}{p-1}q_2-2y\partial_y q_2},
\end{equation}
with
\begin{equation}\label{def psi}
\psi_{d,\nu}=p\kappa^*_1(d,\nu)^{p-1}-\frac{2(p+1)}{(p-1)^2},
\end{equation}
and
\begin{equation}\label{def f nu,d}
f_{d,\nu}(q_1)=|\kappa^*_1(d,\nu)+q_1|^{p-1}(\kappa^*_1(d,\nu)+q_1)-\kappa^*_1(d,\nu)^p-p\kappa^*_1(d,\nu)^{p-1}q_1.
\end{equation}
We introduce
\begin{equation}\label{def varphi}
    \varphi(q,r)=\int^1_{-1}\left(q_1'r_1'(1-y^2)-\psi_{d,\nu} q_1r_1+q_2r_2\right)\rho dy,
\end{equation}
where $\psi_{d,\nu}$ is defined in \eqref{def psi}. We have the following useful results:

\begin{lemma}\label{lemma diff ineq para}
\begin{enumerate}[label=(\roman*)]
\item (Dynamics of the parameters) There exists $\bar C>0$ such that for $s_0$ large enough and for all $s\in [s_0,\bar s)$, we have

\begin{equation}\label{diff inequality xi nu d original}
    \frac{|d'|+|\nu'-\nu|}{1-d^{*2}}\leq \bar{C}\left(||q||^2_{\mathcal{H}}+||q||_{\mathcal{H}}\frac{|\nu|}{1-d^{*2}}+e^{-s}\right).
\end{equation}\label{dyn of the parameters}
\item We have the following estimate of $\varphi(q,q)$
\begin{equation}
        \left(\frac{1}{2}\varphi(q,q)+R\right)'\leq \frac{C}{s_0^{\frac{1}{4}}}  \varphi(q,q)+\left(Ce^{-s}-\frac{4}{p-1}\right)\int^1_{-1}q_2^2\frac{\rho}{1-y^2} d y+C e^{-s},
\end{equation}
where 
\begin{equation}\label{def A and R}
R=-\int^{1}_{-1} \mathcal{F}_{\nu,d}(q_1)dy, \text{ with }\mathcal{F}_{\nu,d}(q_1)=\int^{q_1}_0 f_{\nu,d}(\xi)d\ x,
\end{equation}
and $f_{\nu,d}$ is defined in \eqref{def f nu,d}. \label{diff 1/2 A+R estimate}
\item We have an additional estimate
\begin{equation}
\begin{split}
        \frac{d}{ds}\int^1_{-1}q_1q_2 \rho=&\ -\frac{9}{10}\varphi(q,q)+C\int^1_{-1}q_2^2\frac{\rho}{1-y^2}+ C e^{-s}.
\end{split}    
\end{equation}\label{additional estimate}
\end{enumerate}
\end{lemma}

\begin{proof}
\textbf{Proof of \ref{dyn of the parameters}:} We project equation \eqref{eq q} with the projector $\pi^d_\lambda$ \eqref{def pi} for $\lambda = 0$ and $\lambda = 1$, we write
\begin{equation}\label{proj eq q}
\begin{split}
\pi^{d^*}_\lambda(\partial_s q)=&\pi^{d^*}_\lambda(L_{d,\nu}q)+\pi^{d^*}_\lambda \binom{0}{f_{d,\nu}(q_1)}-d'\pi^{d^*}_\lambda(\partial_d \kappa^*)+(\nu'-\nu)\pi^{d^*}_\lambda(\partial_\nu \kappa^*)\\
&+\pi^{d^*}_\lambda\binom{0}{e^{-s}\frac{(N-1)}{r_0+ye^{-s}}\partial_y (\kappa^*_1+q_1)}.     
\end{split}
\end{equation}

Defining $\alpha_\lambda=\pi^{d^*}_\lambda(q)$ for $\lambda=0,1$, where $\pi^{d^*}$ is defined in \eqref{def pi}, we know from \eqref{def w,kappa,d,nu,q} that for all $s\in [s_0,\bar s)$, $\alpha_\lambda(s)=0$, hence $\alpha'_\lambda(s)=0$. Together with \cite[(194)]{MZ2007}, we have
\begin{equation}\label{proj partial_s q}
    |\pi^{d^*}_\lambda(\partial_s q)-\alpha'_\lambda(s)|=|\pi^{d^*}_\lambda(\partial_s q)|\leq \frac{C}{1-d^{*2}}|d^{*'}|\ ||q||_{\mathcal{H}}.
\end{equation}
Next, we write \eqref{def L_d,nu} as
\begin{equation}\label{Ld* and V}
\begin{split}
L_{d,\nu} q&=\binom{q_2}{\mathcal{L}q_1+\psi_{d^*} q_1-\frac{p+3}{p-1}q_2-2y\partial_y q_2}+\binom{0}{p(\kappa^*_1(d,\nu)^{p-1} -\kappa(d^*)^{p-1}) q_1}\\
&=:L_{d^*}q+\binom{0}{V q_1},
\end{split}
\end{equation}
where \begin{equation}\label{def psi d}
\psi_{d^*}=p\kappa(d^*)^{p-1}-\frac{2(p+1)}{(p-1)^2}.
\end{equation}
From \cite[(195)]{MZ2007} and \cite[$(i)$ of Claim 3.2]{MZ2012}, we have
\begin{equation}\label{proj L_d}
    \pi^{d^*}_\lambda(L_{d^*} q)=\lambda \alpha_\lambda=0,
\end{equation}
\begin{equation}\label{proj diff kappa*}
\begin{matrix}
\pi^{d^*}_0(\partial_\nu \kappa^*)=0, & -\frac{C}{1-d^{*2}}\leq\pi^{d^*}_1(\partial_\nu \kappa^*) \leq -\frac{1}{C(1-d^{*2})},\\
-\frac{C}{1-d^{*2}}\leq\pi^{d^*}_0(\partial_d \kappa^*) \leq -\frac{1}{C(1-d^{*2})}, & |\pi^{d^*}_1(\partial_d \kappa^*)|\leq \frac{C}{1-d^{*2}}.
\end{matrix} 
\end{equation}
From similar analysis as in \cite[Proof of Lemma C.2, page 2900]{MZ2012} and \cite[(C.7)]{MZ2012E}, we obtain
\begin{equation}\label{proj V f}
\begin{split}
    &\left|\pi^{d^*}_\lambda \binom{0}{V q_1}\right|\leq \frac{C|\nu|}{1-d^{*2}}||q||_{\mathcal{H}},\\
    &\left|\pi^{d^*}_\lambda \binom{0}{f(q_1)}\right|\leq C||q||^2_\mathcal{H}.    
\end{split}
\end{equation}

Therefore, we see from definitions \eqref{def kappa} and \eqref{def W_lambda,2} of $\kappa(d^*)$ and $W_{\lambda,2}^{d^*}$ that $W_{\lambda,2}^{d^*}\leq C\kappa(d^*)$, Together with definition \eqref{def pi}, we obtain
\begin{equation}\label{proj exp term}
\begin{split}
\left| \pi^{d^*}_\lambda\binom{0}{e^{-s}\frac{(N-1)}{r_0+ye^{-s}}\partial_y (\kappa^*_1(d,\nu)+q_1)}\right|&=\left| (N-1)e^{-s}\int^{1}_{-1} W_{\lambda,2}^{d^*} \frac{1}{r_0+ye^{-s}}\partial_y (\kappa^*_1(d,\nu)+q_1)\rho\right|\\
&\leq \frac{Ce^{-s}}{r_0-e^{-s}}\left| \int^{1}_{-1} \kappa(d^*) \partial_y (\kappa^*_1(d,\nu)+q_1)\rho\right|.
\end{split}
\end{equation}
Using Cauchy-Schwarz inequality, Lemma \ref{estimates in different norms} and \eqref{def H}, we have for the second term
\begin{equation}\label{proj exp term 2}
\left| \int^{1}_{-1} \kappa(d^*) \partial_y q_1(d,\nu)\rho\right|\leq \left(\int^{1}_{-1} \kappa^2(d^*)\frac{\rho}{1-y^2}\right)^{\frac{1}{2}} \left(\int^{1}_{-1}(\partial_y q_1(d,\nu))^2(1-y^2)\rho\right)^{\frac{1}{2}}\leq C ||q||_{\mathcal{H}}.
\end{equation}
For the first term, with the same computations, we obtain
\begin{equation}\label{proj exp term 1}
\begin{split}
\left| \int^{1}_{-1} \kappa(d^*) \partial_y \kappa^*_1(d,\nu)\rho\right| \leq C ||\kappa^*(d,\nu)||_{\mathcal{H}}.
\end{split}
\end{equation}
Using \eqref{kappa* norm lambda} and \eqref{bound lambda-1}, we obtain the following bound on $\kappa^*(d,\nu)$ for $s_0$ is large enough
\begin{equation}\label{kappa* bound}
 ||\kappa^*(d,\nu)||_{\mathcal{H}}\leq C.   
\end{equation}
Combining this with \eqref{consequances modulation technique}, \eqref{proj exp term}, \eqref{proj exp term 2} and \eqref{proj exp term 1} we obtain  for all $s\in [s_0,\bar s)$,

\begin{equation}\label{proj exp}
\begin{split}
\left| \pi^{d^*}_\lambda\binom{0}{e^{-s}\frac{(N-1)}{r_0+ye^{-s}}\partial_y (\kappa^*_1(d,\nu)+q_1)}\right| \leq Ce^{-s}(1+||q||_{\mathcal{H}})\leq Ce^{-s}.
\end{split}
\end{equation}
Combining \eqref{proj eq q}, \eqref{proj partial_s q}, \eqref{Ld* and V} \eqref{proj L_d}, \eqref{proj diff kappa*}, \eqref{proj V f} and \eqref{proj exp}, we have for $\lambda=0$

\begin{equation}\label{d' inequality}
\frac{|d'|}{1-d^{*2}}\leq C\left(\frac{|d^{*'}|+|\nu|}{1-d^{*2}}||q||_{\mathcal{H}}+||q||_{\mathcal{H}}^2+e^{-s}\right).
\end{equation}
With the same computations but for $\lambda=1$ and with \eqref{d' inequality}, we have
\begin{equation}\label{proj lambda=1}
\frac{|\nu'-\nu|}{1-d^{*2}}\leq C\left(\frac{|d^{*'}|+|\nu|}{1-d^{*2}}||q||_{\mathcal{H}}+||q||_{\mathcal{H}}^2+e^{-s}\right).
\end{equation}

Using \eqref{def d*} and \eqref{equiv d and d*}, for $s_0$ large enough, we have the following
\begin{equation}\label{d*/1-d* bound}
   \frac{|d^{*'}|}{1-d^{*^2}}\leq C^*\frac{|d'|+|\nu-\nu'|+|\nu|}{1-d^{*2}}.
\end{equation}

Therefore combining \eqref{d' inequality}, \eqref{proj lambda=1}, \eqref{equiv d and d*}, \eqref{d*/1-d* bound}, with the \eqref{consequances modulation technique} for $s_0$, we obtain \eqref{diff inequality xi nu d original}, and concludes the proof of item \ref{dyn of the parameters}.

\bigskip
\textbf{Proof of \ref{diff 1/2 A+R estimate}:} One may see that with similar computation as in \cite[(C.25), (C.26) and (C.27)]{MZ2012} that we have
\begin{equation}
    \left|\frac{1}{2}\partial_s \varphi(q,q)-\varphi(\partial_s q, q)\right|\leq C\frac{|d'|+|\nu'|}{1-d^{*2}}||q||^2_{\mathcal{H}}.
\end{equation}
With \eqref{diff inequality xi nu d original}, we have
\begin{equation}\label{A-varphi estimate}
    \left|\frac{1}{2}\partial_s \varphi(q,q)-\varphi(\partial_s q, q)\right|\leq C\left(||q||^2_{\mathcal{H}}+(1+||q||_{\mathcal{H}})\frac{|\nu|}{1-d^{*2}}+e^{-s}\right)||q||^2_{\mathcal{H}}.
\end{equation}
Together with \eqref{consequances modulation technique} and \eqref{equiv d and d*}, taking $s_0$ large enough, we obtain
\begin{equation}\label{A-varphi estimate epsilon}
    \left|\frac{1}{2}\partial_s \varphi(q,q)-\varphi(\partial_s q, q)\right|\leq \frac{C}{s_0^{\frac{1}{4}}} ||q||^2_{\mathcal{H}}.
\end{equation}
Using \eqref{eq q}, we have that
\begin{equation}\label{varphi d^* eq q}
\begin{split}
\varphi(\partial_s q,q)=&\varphi(L_{d,\nu}q,q)+\varphi\left( \binom{0}{f_{d,\nu}(q_1)},q\right)-d'\varphi(\partial_d \kappa^*,q)-(\nu'-\nu)\varphi(\partial_\nu \kappa^*,q)\\
&+\varphi\left(\binom{0}{e^{-s}\frac{(N-1)}{r_0+ye^{-s}}\partial_y (\kappa^*_1+q_1)},q\right),  
\end{split}
\end{equation}
where $\varphi$ is defined in \eqref{def varphi}.
We have from \cite[(208)]{MZ2007}, the following estimate
\begin{equation}\label{varphi L d,nu estimate}
\varphi(L_{d,\nu}q,q)=-\frac{4}{p-1}\int^1_{-1}q_2^2\frac{\rho}{1-y^2} d y,
\end{equation}
With similar computations as in \cite[(C.29) and (C.31)]{MZ2012}, we obtain
\begin{equation}\label{varphi diff s kappa* and R'+f estimate}
\begin{split}
    &|d'\varphi(\partial_d \kappa^*,q)+(\nu'-\nu)\varphi(\partial_\nu \kappa^*,q)|\leq C\frac{|d'|+|\nu'-\nu|}{1-d^{*2}}||q||_{\mathcal{H}},\\
    &\left|R'+\varphi\left( \binom{0}{f_{d,\nu}(q_1)},q\right)\right|\leq C\frac{|d'|+|\nu'|+|\nu|}{1-d^*}||q||^2_{\mathcal{H}}.
    \end{split}
\end{equation}
With \eqref{diff inequality xi nu d original} and \eqref{consequances modulation technique}, we obtain by taking $s_0$ large enough
\begin{equation}\label{varphi diff s kappa* and R'+f estimate epsilon}
\begin{split}
    &|d'\varphi(\partial_d \kappa^*,q)+(\nu'-\nu)\varphi(\partial_\nu \kappa^*,q)|\leq \frac{C}{s_0^{\frac{1}{4}}}(||q||_{\mathcal{H}}^2+ e^{-s}),\\
    &\left|R'+\varphi\left( \binom{0}{f_{d,\nu}(q_1)},q\right)\right|\leq \frac{C}{s_0^{\frac{1}{4}}}||q||^2_{\mathcal{H}}.
    \end{split}
\end{equation}
For the last term, from the definition of $\varphi$ in \eqref{def varphi} and using \eqref{hardy-sobolev identity} \eqref{kappa* bound} and \eqref{consequances modulation technique}, we have
\begin{equation}\label{varphi extra term estimate epsilon}
\begin{split}
\left|\varphi\left(\binom{0}{e^{-s}\frac{(N-1)}{r_0+ye^{-s}}\partial_y (\kappa^*_1+q_1)},q\right)\right|& =(N-1)e^{-s} \left|\int^1_{-1}\frac{1}{r_0+ye^{-s}}\partial_y (\kappa^*_1+q_1)q_2 \rho\right|\\
&\leq e^{-s}\frac{(N-1)}{r_0-e^{-s}}\left(\left|\int^1_{-1}\partial_y \kappa^*_1 q_2 \rho\right|+ \left|\int^1_{-1}\partial_y q_1 q_2  \rho\right|\right)\\
&\leq Ce^{-s}\left(||\kappa^*_1 ||_{\mathcal{H}_0}+ ||q_1||_{\mathcal{H}_0} \right)\left(\int^1_{-1} q_2^2 \frac{\rho}{1-y^2}\right)^{\frac{1}{2}}\\
&\leq Ce^{-s}\left(1+ ||q||_{\mathcal{H}} \right)\left(1+\int^1_{-1} q_2^2 \frac{\rho}{1-y^2}\right).\\
&\leq Ce^{-s}\left(1+\int^1_{-1} q_2^2 \frac{\rho}{1-y^2}\right).
\end{split}
\end{equation}
Gathering the estimates \eqref{A-varphi estimate epsilon}, \eqref{varphi d^* eq q}, \eqref{varphi L d,nu estimate}, \eqref{varphi diff s kappa* and R'+f estimate epsilon}, and \eqref{varphi extra term estimate epsilon}, we have for $s_0$ sufficiently large that
\begin{equation}
\begin{split}
        \left(\frac{1}{2}\varphi(q,q)+R\right)'=&\frac{1}{2}\partial_s \varphi(q,q)-\varphi(\partial_s q,q)+\varphi(\partial_s q,q)+R'\\
        \leq &\frac{C}{s_0^{\frac{1}{4}}} ||q||_{\mathcal{H}}^2+\left(Ce^{-s}-\frac{4}{p-1}\right)\int^1_{-1}q_2^2\frac{\rho}{1-y^2} d y+C e^{-s}.
\end{split}
\end{equation}
with \eqref{ineq norm q2 & A}, we obtain \ref{diff 1/2 A+R estimate}.

\bigskip
\textbf{Proof of \ref{additional estimate}: } From \eqref{eq q}, we have
\begin{equation}\label{extra term ineq}
\begin{split}
        \frac{d}{ds}\int^1_{-1}q_1q_2 \rho=&\int^1_{-1}\partial_s q_1q_2 \rho+\int^1_{-1} q_1\partial_sq_2 \rho\\
        =&\int^1_{-1}q_2^2\rho-d'\int^1_{-1} \partial_d\kappa^*(d,\nu).(q_2,q_1)\rho -(\nu'-\nu)\int^1_{-1} \partial_\nu\kappa^*(d,\nu).(q_2,q_1)\rho\\
        &+ \int^1_{-1} q_1(\mathcal{L}q_1+\psi_{d,\nu} q_1-\frac{p+3}{p-1}q_2-2y\partial_y q_2)\rho+\int^1_{-1} q_1 f_{d,\nu}(q_1)\rho\\
        &+ e^{-s}\int^1_{-1}\frac{(N-1)}{r_0+ye^{-s}} q_1\partial_y (\kappa^*_1(d,\nu)+q_1)\rho,
\end{split}    
\end{equation}
where the dot “.” stands for the inner product coordinate by coordinate. As in \cite[page 2904]{MZ2012}, we have
\begin{equation}\label{extra term some ineq}
\begin{split}
        &\int^1_{-1}q_2^2\rho\leq \int^1_{-1}q_2^2\frac{\rho}{1-y^2},\\
        &\int^1_{-1} q_1(\mathcal{L}q_1+\psi_{d,\nu} q_1)\rho \leq -\varphi(q,q)+\int^1_{-1}q_2^2\frac{\rho}{1-y^2},\\
        &\left|\int^1_{-1} -\frac{p+3}{p-1}q_1q_2-2yq_1\partial_y q_2\rho\right| \leq \frac{1}{100}\varphi(q,q)+C\int^1_{-1}q_2^2\frac{\rho}{1-y^2},\\
        &\left|\int^1_{-1} q_1 f_{d,\nu}(q_1)\right| \leq \frac{1}{100}\varphi(q,q).
\end{split}    
\end{equation}
With similar computation as in \cite[page 2904]{MZ2012}, we have
\begin{equation}
            \left|d'\int^1_{-1} \partial_d\kappa^*(d,\nu).(q_2,q_1)\rho +(\nu'-\nu)\int^1_{-1} \partial_\nu\kappa^*(d,\nu).(q_2,q_1)\rho\right|\leq C\frac{||q||_{\mathcal{H}}}{1-d^{*2}}(|d'|+|\nu'-\nu|).
\end{equation}
Together with \eqref{diff inequality xi nu d original}, \eqref{consequances modulation technique}, \eqref{equiv d and d*} and \eqref{ineq norm q2 & A}, we obtain for $s_0$ sufficiently large

\begin{equation}\label{extra term kappa* ineq}
\begin{split}
            \left|d'\int^1_{-1} \partial_d\kappa^*(d,\nu).(q_2,q_1)\rho \right.&\left.+(\nu'-\nu)\int^1_{-1} \partial_\nu\kappa^*(d,\nu).(q_2,q_1)\rho \right| \\
            &\leq C||q||_{\mathcal{H}}^2\left(||q||_{\mathcal{H}}+\frac{|\nu|}{1-d^{*2}}\right)+Ce^{-s}||q||_{\mathcal{H}}\\
            &\leq C||q||_{\mathcal{H}}^2+Ce^{-s}\\
            &\leq \frac{1}{100}\varphi(q,q)+Ce^{-s}.
\end{split}
\end{equation}
For the last term, we proceed as the following

\begin{equation}\label{extra term exp ineq}
\begin{split}
    \left|e^{-s}\frac{(N-1)}{r_0+ye^{-s}}\int^1_{-1} q_1\partial_y (\kappa^*_1(d,\nu)+q_1)\rho\right|&\\
    \leq \frac{Ce^{-s}}{r_0-e^{-s}}\left(\left|\int^1_{-1} q_1\partial_y \kappa^*_1(d,\nu)\rho\right|\right. & \left.+\left| \int^1_{-1} q_1\partial_y q_1 \rho\right|\right) =:\frac{Ce^{-s}}{r_0-e^{-s}}(I_1+I_2).
\end{split}
\end{equation}
Using Cauchy-Schwarz inequality, Lemma \ref{estimates in different norms} and \eqref{kappa* bound}, we obtain
$$I_1=\left|\int^1_{-1} q_1\partial_y \kappa^*_1(d,\nu)\rho\right|\leq ||q_1||_{L^2_{\frac{\rho}{1-y^2}}}|| \kappa^{*}_1(d,\nu)||_{\mathcal{H}_0}\leq C||q||_{\mathcal{H}}.$$
With the same computations, we obtain for the second integral
$$I_2\leq C||q||_{\mathcal{H}}^2.$$
Therefore for $s_0$ large enough with \eqref{extra term exp ineq} and \eqref{consequances modulation technique}becomes
\begin{equation}\label{extra term exp}
\begin{split}
    \left|e^{-s}\frac{(N-1)}{r_0+ye^{-s}}\int^1_{-1} q_1\partial_y (\kappa^*_1(d,\nu)+q_1)\rho\right|\leq Ce^{-s} ||q||_{\mathcal{H}}(1+||q||_{\mathcal{H}}) \leq C e^{-s}.
\end{split}
\end{equation}
Combining \eqref{extra term ineq}, \eqref{extra term some ineq}, \eqref{extra term kappa* ineq} and \eqref{extra term exp}, we obtain \ref{additional estimate}. This concludes the proof of Lemma \ref{lemma diff ineq para}.

\end{proof}

From the differential inequalities in lemma \ref{lemma diff ineq para}, we can derive some energy functional for equation \eqref{eq q} which is easily controllable and equivalent to the norm squared. More precisely, we claim the following:

\begin{cor}[Some energy functional]
There exist constants $C^*>0$ and $\delta \in [0,1)$ such that for all $ s\in [s_0,\bar s)$, we have
\begin{equation}\label{h2' ineq}
\frac{1}{C^*}||q||^2_{\mathcal{H}}\leq h \leq C^*||q||^2_{\mathcal{H}} \text{ and }h'\leq  -\delta h+Ce^{-s},
\end{equation}
where
\begin{equation*}
    h=\frac{1}{2}\varphi(q,q)-\int^1_{-1} \mathcal{F}(q_1)\rho+\eta \int^1_{-1}q_1q_2\rho,
\end{equation*}
with $\mathcal{F}(q_1)$ is defined in \eqref{def A and R} and $\eta=\frac{3}{C^*(p-1)}$.  
\end{cor}
\begin{proof}
For the first inequality, with similar computation as \cite[(C.21)]{MZ2012}, we have
\begin{equation}\label{ineq h & A }
    \left|h-\frac{1}{2}\varphi(q,q)\right|\leq \frac{1}{100}\varphi(q,q).
\end{equation}
Combining this with \eqref{ineq norm q2 & A}, we obtain the first inequality.

For the second inequality, we have from definition of $\varphi$ that
\begin{equation}
\begin{split}
        h'&\leq C h\left(\frac{1}{s_0^{\frac{1}{4}}} -\eta\frac{9}{10}\right)+\left(C^*(e^{-s}+\eta)-\frac{4}{p-1}\right)\int^1_{-1}q_2^2\frac{\rho}{1-y^2} d y+C(1+\eta) e^{-s}.
\end{split}
\end{equation}
By taking $s_0$ sufficiently large, $\eta=\frac{3}{C^*(p-1)}$, and $s_0$ large enough, there exists $\delta\in (0,1)$ such that
\begin{equation}
        h'\leq  -\delta h+Ce^{-s}.
\end{equation}

Obtaining the second inequality, we conclude the proof.
\end{proof}

We are now in a position to prove Proposition \ref{Dynamics of the parameters}:
\begin{proof}[Proof of Proposition \ref{Dynamics of the parameters}]
We have from the second inequality of \eqref{h2' ineq} that
$$\left( e^{\delta s} h\right)' \leq Ce^{(\delta-1) s}.$$
Therefore for all $s\in [s_0,\bar s)$, we have
\begin{equation}
\begin{split}
h(s)&\leq \frac{Ce^{-s_0}}{1-\delta}\left[e^{-\delta(s- s_0)}-e^{-(s-s_0)}\right]+e^{-\delta(s- s_0)}h(s_0).\\
&\leq Ce^{-\delta(s- s_0)-s_0}+e^{-\delta(s- s_0)}h(s_0),\\
\end{split}
\end{equation}
which concludes the proof of Proposition \ref{Dynamics of the parameters}.
\end{proof}

\section{Trapping in some shrinking set}\label{section shrink set}
In this section, we prove Proposition \ref{sol conv solition} by trapping $\nu, d$ and $q$, defined in \eqref{def w,kappa,d,nu,q} and right before and in \eqref{q}, in a set shrinking to zero. We proceed in two steps: We use the differential equality and inequalities  obtained earlier combined with a topological argument. Then, we conclude the proof of Proposition \ref{sol conv solition} using a result obtained in \cite{MZ2012}.

\subsection{Reduction thanks to some shrinking set}\label{subsec shriking set}
We introduce the following set $V_{A,\dhz}(s)$:
\begin{defn}[A set shrinking to zero]
    For  $\dhz\in (-1,1)$, $A>0$, $s_0\in \R$ and $s\geq s_0$, we define $V_{A,\dhz}(s)$ as the set of $(\nu(s),d(s),q(s))\in (-1,1)^2\times\mathcal{H}$ such that
    \begin{equation}\label{ineq V_A,d_0}
        \begin{split}
            &|\nu(s)|\leq A e^{-\delta(s-s_0)-s_0},\\
            &|d(s)-\dhz|\leq A e^{-\delta(s-s_0)-s_0},\\
            &||q(s)||^2_{\mathcal{H}}\leq A e^{-\delta(s-s_0)-s_0},\\
        \end{split}
    \end{equation}
    where $\delta\in (0,1)$ has already been fixed in Proposition \ref{Dynamics of the parameters}.
\end{defn}
We have the following proposition giving us, for well-chosen initial data, the existence of a solution in the shrinking set: 
\begin{prp}[Existence of a solution in $V_{A,\dhz}(s)$]\label{initial data}
    For all $r_0>0$ and $\dhz\in (-1,1)$, there exists $A_0>0$, such that for  all $A\geq A_0$, there exists $s_0=s_0(r_0,\dhz,A)>-\log r_0$, such that there exist $\dz\in (-1,1), \nz>-1+|\dz|$ with
    $$|\dz-\dhz|\leq Ae^{-s_0} \text{ and } |\nz|\leq Ae^{-s_0},$$
    so that $(\nu(s),d(s), q(s))\in V_{A,\dhz}(s)$, for  all $s\geq s_0$, where $q,d$ and $\nu$ are defined in \eqref{def w,kappa,d,nu,q} and right before, and $w$ is the solution of equation \eqref{eq similary} with initial data \eqref{initial data soliton}.
\end{prp}

\begin{proof}
   Consider $r_0>0, \dhz\in (-1,1), A\geq A_0$ and $s_0\geq -\log r_0$, where $A_0$ and $s_0$ will be chosen later. For any $(\dz,\nz)\in \R^2$ such that
   \begin{equation}\label{initial bar d,nu in V_A,d_0}
       |\dz-\dhz|\leq A e^{-s_0} \text{ and } |\nu_0|\leq A e^{-s_0},
   \end{equation}
   it holds that $\dz\in (-1,1)$, and $\nz>-1+|\dz|$ for $s_0$ large enough. Hence, initial data in \eqref{initial data soliton} are well-defined and belongs to $\mathcal{H}$. From \eqref{nu,d(s0)=nu0,d_0} and \eqref{initial data soliton}, one may see that $$(\nu(s_0),d(s_0),q(s_0))=(\nz,\dz,0)\in V_{A,\dhz}(s_0).$$
    We assume $(\nu(s),d(s),q(s))$ does not belong to $V_{A,\dhz}(s)$ for some $s\in [s_0,\bar s)$, where $\bar s$ is defined above \eqref{def w,kappa,d,nu,q}. By continuity, there exists a maximal $s^*=s^*(\dz,\nz)\in [s_0,\bar s]$ such that for all $s\in [s_0,s^*]$, we have $(\nu(s),d(s), q(s))\in V_{A,\dhz}(s)$.
   
   If $s^*=+\infty$ for some $(\dz,\nz)$ such that \eqref{initial bar d,nu in V_A,d_0} holds, then we get to the conclusion of Proposition \ref{initial data}.
   
   If $s^*<+\infty$ for all $(\dz,\nz)$ such that \eqref{initial bar d,nu in V_A,d_0}, we will find a contradiction below. Two cases may happen if $s^*<+\infty$:

   \textbf{Case 1:} If $s^*=\bar s$, then $\bar s<+\infty$. As said in the statement after \eqref{equiv d and d*}, one of the inequalities of \eqref{consequances modulation technique} becomes an equality, i.e.
    \begin{equation}\label{eq on s_0}
    \text{either }\ \frac{|\nu(\bar s)|}{1-|d(\bar s)|}= s_0^{-\frac{1}{4}}, \text{ or }\ ||q(\bar s)||_{\mathcal{H}}= s_0^{-1/2},
    \end{equation}
    From the fact of $(\nu,d,q)(\bar s)\in V_{A,\dhz}(\bar s)$, we have 
    \begin{align}
        &||q(\bar s)||_{\mathcal{H}}^2\leq Ae^{-\delta(\bar s-s_0)-s_0},\label{q(bar s) in V_A,d_0}\\
        &|d(\bar s)-\dhz|\leq A e^{-\delta(\bar s-s_0)-s_0},\label{d(bar s) in V_A,d_0}\\
        &|\nu(\bar s)|\leq A e^{-\delta(\bar s-s_0)-s_0}.\label{nu(bar s) in V_A,d_0}
    \end{align}
Hence, for $s_0$ large enough, we have
$$|d(\bar s)-\dhz|\leq \frac{1-|\dhz|}{2}.$$
In particular, we have
$$|d(\bar s)|\leq \frac{1+|\dhz|}{2}.$$
Hence
 $$1-|d(\bar s)|\geq \frac{1-|\dhz|}{2}.$$
Using \eqref{nu(bar s) in V_A,d_0}, for $s_0$ large enough, we obtain
\begin{equation}\label{ineq nu/1-d}
\frac{|\nu(\bar s)|}{1-|d(\bar s)|}\leq \frac{2A}{1-|\dhz|}e^{-\delta(\bar s-s_0)-s_0}<s_0^{-\frac{1}{4}}.
\end{equation}
Together with \eqref{q(bar s) in V_A,d_0}, for $s_0$ large enough, we have a contradiction with \eqref{eq on s_0}.

\textbf{Case 2:} If $s^*<\bar s$, then by continuity and by definition of $s^*$, we have either
\begin{equation}\label{eq s^*}
|\nu(s^*)|= A e^{-\delta(s^*-s_0)-s_0}, \text{ or } |d(s^*)-\dhz|= A e^{-\delta(s^*-s_0)-s_0}, \text{ or } ||q(s^*)||^2_{\mathcal{H}}= A e^{-\delta(s^*-s_0)-s_0}.
\end{equation}
Together with Proposition \ref{Dynamics of the parameters}, taking $A>C^*$, we exclude the third equality. Therefore, we have only two possibilities:
\begin{align}
    &\text{either }|d(s^*)-\dhz|= A e^{-\delta(s^*-s_0)-s_0},\label{eq d s*}\\
    &\text{or }|\nu(s^*)|= A e^{-\delta(s^*-s_0)-s_0}\label{eq nu s*}.
\end{align}

\textbf{Case 2.1:} We assume that \eqref{eq d s*} holds. From \eqref{diff inequality xi nu d original}, Proposition \ref{Dynamics of the parameters}, \eqref{equiv d and d*}, together with the fact that $0<\delta<1$, we have for all $s\in [s_0,s^*]$
$$|d'(s)|\leq \bar C\left((C^*+1)e^{-\delta(s- s_0)-s_0}+\sqrt{C^*}Ae^{-\frac{\delta}{2}(s- s_0)-\frac{s_0}{2}} e^{-\delta(s^*-s_0)-s_0}\right).$$
For $s_0$ large enough, we have
$$|d'(s^*)|\leq \bar C(C^*+2)e^{-\delta(s^*- s_0)-s_0}.$$
From \eqref{eq d s*}, if
$$d(s^*)=\dhz+Ae^{-\delta(s^*- s_0)-s_0},$$
taking $A> \max(C^*,\frac{\overline{C}(C^*+2)}{\delta})$, we have
$$\frac{d}{d s}[\dhz+Ae^{-\delta(s- s_0)-s_0}]_{s=s^*}=-\delta Ae^{-\delta(s^*- s_0)-s_0}<d'(s^*).$$
Therefore, the curve $s\mapsto d(s)$ transversely crosses the curve $s\mapsto d_0+Ae^{-\delta(s- s_0)-s_0}$ at $s=s^*$. The same arguments can be used for the opposite case, when $d(s^*)=\dhz-Ae^{-\delta(s^*- s_0)-s_0}$.

%$$|d(s^*)-\dz|=|\int^{s^*}_{s_0}d'(s)|\leq \int^{s^*}_{s_0}|d'(s)|\leq \frac{\bar C(C^*+2)}{\delta}(e^{-s_0}-e^{-\delta(s^*- s_0)-s_0})\leq \frac{\bar C(C^*+2)}{\delta}e^{-s_0},$$
%for $s_0$ large enough for the last inequality. Thus
%$$|d(s^*)- \dhz|\leq \left(\frac{\bar C(C^*+2)}{\delta}+1\right)e^{-s_0},$$

%Taking $A_0=\frac{\bar C(2C^*+1)}{\delta}$, this contradicts the first inequality of \eqref{eq s^*}.

\textbf{Case 2.2:} We assume that \eqref{eq nu s*} holds. Again, from \eqref{diff inequality xi nu d original}, Proposition \ref{Dynamics of the parameters} and the fact that $|d^*|<1$, we have for all $s\in [s_0,s^*]$
$$|\nu'(s)-\nu(s)|\leq \bar C\left((C^*+1)e^{-\delta(s- s_0)-s_0}+\sqrt{C^*}Ae^{-\frac{\delta}{2}(s- s_0)-\frac{s_0}{2}} e^{-\delta(s^*-s_0)-s_0}\right).$$
For $s_0$ large enough, we have that
$$|\nu'(s^*)-\nu(s^*)|\leq \bar C(C^*+2)e^{-\delta(s^*- s_0)-s_0}.$$
From \eqref{eq d s*}, if
$$\nu(s^*)=Ae^{-\delta(s^*- s_0)-s_0},$$
then
$$\nu'(s^*)\geq \nu(s^*)-\bar C(C^*+2)e^{-\delta(s^*- s_0)-s_0}= (A-\bar C(C^*+2))e^{-\delta(s^*- s_0)-s_0}.$$
Since $\delta<1$, taking again $A> \max(C^*,\frac{\overline{C}(C^*+2)}{\delta})$, we have
$$\frac{d}{d s}[Ae^{-\delta(s- s_0)-s_0}]_{s=s^*}<\nu'(s^*).$$
Therefore, the curve $s\mapsto \nu(s)$ transversely crosses the curve $s\mapsto Ae^{-\delta(s- s_0)-s_0}$ at $s=s^*$. The same arguments can be used for the opposite case, when $\nu(s^*)=-Ae^{-\delta(s^*- s_0)-s_0}$.

\textbf{Conclusion:} Take $A_0=\max(C^*,\frac{\overline{C}(C^*+2)}{\delta})+1$. Whether Case 2.1 or Case 2.2 holds, we have just proved that the flow $s\mapsto (d(s),\nu(s))$ is transverse outgoing on the surface $s\mapsto( \dhz \pm Ae^{-\delta(s- s_0)-s_0}, \pm Ae^{-\delta(s- s_0)-s_0})$ at the crossing point $s=s^*(\dz, \nz)$. In particular, $(\dz, \nz)\mapsto s^*(\dz, \nz)$ is continuous. Let us introduce the following change of variables
$$(\bar \nu_0,\bar d_0)=\left(\frac{\nu_0}{Ae^{-s_0}},\frac{d_0-\dhz}{Ae^{-s_0}}\right),$$
and the following function
\begin{equation}
\begin{array}{ccccc}
\Phi & : & [-1,1]^2  & \to & \partial [-1,1]^2 \\
 & & (\bar \nu_0,\bar d_0) & \mapsto &(\frac{\nu(s^*)}{Ae^{-\delta(s^*- s_0)-s_0}},\frac{d(s^*)-\dhz}{Ae^{-\delta(s^*- s_0)-s_0}}). \\
\end{array}
\end{equation}
We claim the following properties for $\Phi$:
\begin{enumerate}[label=(\roman*)]
    \item It is well-defined from \eqref{eq d s*} and \eqref{eq nu s*}.
    \item It is continuous from the outgoing transverse property.
    \item If $(\bar \nu_0,\bar d_0)\in \partial [-1,1]^2$, then $(\nz,\dz)\in \partial \left([-Ae^{-s_0},Ae^{-s_0}]\times [\dhz-Ae^{-s_0},\dhz+Ae^{-s_0}]\right)$, which implies $s^*(\nz,\dz) =s_0$, from the outgoing transverse property. Together with \eqref{nu,d(s0)=nu0,d_0}, we obtain that $\Phi_{\partial [-1,1]^2}=Id_{\partial [-1,1]^2}$.
\end{enumerate}
From index theory, we know that such a function with the previous properties does not exist, and we have our contradiction, which concludes the proof of Proposition \ref{initial data}.
\end{proof}

\subsection{Proof of Proposition \ref{sol conv solition}}\label{proof of prp}
Before concluding the proof of Proposition \ref{sol conv solition}, let us first recall from \cite[Lemma A.2]{MZ2012} the following continuity result for the solitons $\kappa^*$:
\begin{lemma}[Continuity of $\kappa^*$]\label{continuity kappa*}
For all $B \geq 2$, there exists $C(B)> 0$ such that if $(d_1,\nu_1)$ and $ (d_2,\nu_2)$ satisfy
$$\frac{\nu_1}{1-|d_1|},\frac{\nu_2}{1-|d_2|}\in \left[-1+\frac{1}{B},B\right],$$
then 
$$||\kappa^*(d_1,\nu_1)-\kappa^*(d_1,\nu_1)||_{\mathcal{H}}\leq C(B)\left(\left|\frac{\nu_1}{1-|d_1|}-\frac{\nu_2}{1-|d_2|}\right|+|\arg \tanh d_1-\arg \tanh d_2|\right).$$

\end{lemma}

Let us now give the proof of Proposition \ref{sol conv solition}.
\begin{proof}[Proof of Proposition \ref{sol conv solition}]

Consider $\delta\in(0,1)$ defined in Proposition \ref{Dynamics of the parameters}. Then consider arbitrary $r_0$ and $\dhz\in (-1,1)$. Note that $\kappa^*(\dhz,0)=\binom{\kappa(\dhz)}{0}$. Then, from Proposition \ref{initial data}, there exists $A>0$, $s_0(r_0,\dhz,A)\geq -\log r_0$, $\dz, \nz$, such that the solution $w(y,s)$ of \eqref{eq similary} with initial data given in \eqref{initial data soliton} satisfies $(\nu(s),d(s),q(s))\in V_{A,\dhz}(s)$ for all $s\geq s_0$, where $q,d$ and $\nu$ are defined in \eqref{def w,kappa,d,nu,q} and right before. Therefore,
\begin{equation}\label{conv soliton}
\begin{split}
\left|\left|\binom{w(s)}{\partial_s w(s)}-\binom{\kappa(\dhz)}{0}\right|\right|_{\mathcal{H}}&\leq ||q(s)||_{\mathcal{H}}+||\kappa^*(d,\nu)-\kappa^*(\dhz,0)||_{\mathcal{H}}\\
&\leq \sqrt{A}e^{-\frac{\delta}{2}(s-s_0)-s_0}+||\kappa^*(d,\nu)-\kappa^*(\dhz,0)||_{\mathcal{H}} .   
\end{split}
\end{equation}
Moreover, we have from Proposition \ref{initial data} that 

$$|\nu(s)|\leq A e^{-\delta(s-s_0)-s_0}, |d(s)-\dhz|\leq A e^{-\delta(s-s_0)-s_0}$$
With the same computations as for \eqref{ineq nu/1-d}, for $s_0$ large enough, we have for all $s\geq s_0$,
$$\frac{\nu(s)}{1-|d(s)|}\leq \frac{2A}{1-|\dhz|}e^{-\delta(s-s_0)-s_0}.$$
Thus, for $s_0$ large enough, it holds that
$$\frac{|\nu(s)|}{1-|d(s)|}\in \left[-\frac{1}{2},2\right].$$
Using the fact that $\arg \tanh(x)=\frac{1}{2}\ln\left(\frac{1+x}{1-x}\right)$, for all $x\in(-1,1)$, then applying Lemma \ref{continuity kappa*} with $B=2$, for $s_0$ large enough, we write from \eqref{conv soliton}, for all $s\geq s_0$:
\begin{equation}
\begin{split}
\left|\left|\binom{w(s)}{\partial_s w(s)}-\binom{\kappa(\dhz)}{0}\right|\right|_{\mathcal{H}}&\leq \sqrt{A}e^{-\frac{\delta}{2}(s-s_0)-s_0}+C\frac{|\nu(s)|}{1-|d(s)|}+C|\arg \tanh d(s)-\arg \tanh \dhz|\\
&\leq \sqrt{A}e^{-\frac{\delta}{2}(s-s_0)-s_0}+C\frac{2A}{1-|\dhz|}e^{-\delta(s-s_0)-s_0}+C(\dhz)Ae^{-\delta(s-s_0)-s_0}\\
&\leq C(A,\dhz)e^{-\delta(s-s_0)-s_0},
\end{split}
\end{equation}
which concludes the proof of Proposition \ref{sol conv solition}.
\end{proof}

\section{Proof of Theorem \ref{Thm}}\label{proof of thm}
Now, we have first found our initial data for $\eqref{eq similary}$. The aim of this section is to obtain the initial data for \eqref{eq radial} and conclude the proof of Theorem \ref{Thm}.
Consider $\delta\in (0,1)$ given in Proposition \ref{sol conv solition}, $r_0>0$ and $\dhz\in (-1,1)$. From Proposition \ref{sol conv solition}, we have $s_0\geq -\log r_0$  and $\dz, \nz\in \R$ such that equation \eqref{eq similary}, with initial data \eqref{initial data soliton} has a solution that behaves as in \eqref{profile similary variables}. We introduce $T_0=e^{-s_0}$. Using \eqref{similarity variables} in a reverse way, we define $u$ as the following
\begin{align}\label{def u similarty var}
    &u(r,t)=(T_0-t)^{-\frac{2}{p-1}}w(y,s), \text{ where } y=\frac{r-r_0}{T_0-t}, s=-\log(T_0-t).
\end{align}
Since $w$ is well-defined for $(y,s)\in (-1,1) \times [s_0,+\infty)$, this formula allows us to define a function $u=u(r,t)$ in the backward light cone
\begin{equation}\label{Gamma r_0}
\Gamma(r_0)=\mathcal{C}_{r_0,T_0,1} \cap \{t \geq 0\},
\end{equation}
where $\mathcal{C}_{r_0,T_0,\delta_0}$ is defined right after \eqref{def non charac}. Naturally, $u$ is a solution of equation \eqref{eq radial} in that cone for some initial data, which will be determined shortly afterwards. Note that our cone is included in $\{r>0\}$ since $s_0>-\log r_0$.

In the next step of the proof, we will define the solution $u$ outside the cone $\Gamma(r_0)$. In order to do so, we will first determine the value of the initial data for $(u,\partial_t u)$ at the basis of the cone. Then, extend it to $\R^+$.

Let us then write the value of $(u,\partial_t u)$ at $t=0$ in the basis of the cone $\Gamma(r_0)$. Using \eqref{eq radial}, and the definition of $\kappa^*_1$ and $\kappa^*_2$ in \eqref{kappa*1} and \eqref{kappa*2}, we see that
\begin{equation*}
    \begin{split}
    u(r,0)&=T_0^{-\frac{2}{p-1}}w(y,s_0)=T_0^{-\frac{2}{p-1}}\kappa^*_1(\dz, \nz ,\frac{r-r_0}{T_0}),\label{proof initial}\\
    \partial_t u(r,0)&=T_0^{-\frac{p+1}{p-1}}(\frac{2}{p-1}w(y,s_0)+y.\partial_y w(y,s_0)+ \partial_s w(y,s_0))\\
    &=T_0^{-\frac{p+1}{p-1}}\left(\frac{2}{p-1}\kappa^*_1(\dz, \nz ,y)+y.\partial_y\kappa^*_1(\dz, \nz ,y)+ \kappa^*_2(\dz, \nz ,y)\right)\\
    &=\frac{2}{p-1}T_0^{-\frac{p+1}{p-1}}\frac{\kappa^*_1(\dz, \nz ,y)}{1+\dz y+\nz}.\\
    \end{split}
\end{equation*}    
Together with \eqref{def u similarty var}, we obtain
\begin{equation}\label{singular initial data}
\begin{split}
    &u(r,0)=T_0^{-\frac{2}{p-1}}\kappa^*_1\left(\dz, \nz ,\frac{r-r_0}{T_0}\right)=\uu(\dz,\nz,r_0,T_0,r,0),\\
    &\partial_tu(r,0)=T_0^{-\frac{2}{p-1}}\frac{2}{p-1}\frac{\kappa^*_1\left(\dz, \nz ,\frac{r-r_0}{T_0}\right)}{T_0(1+\nz)+\dz (r-r_0)}=\frac{1}{1+\nz}\partial_t\uu(\dz,\nz,r_0,T_0,r,0),
\end{split}
\end{equation}
using \eqref{kappa*1} and $\uu$ is given in \eqref{def uu}. Let us now define $u$ outside the cone $\Gamma(r_0)$. For that, we will first extend the definition of the initial data outside the basis of the cone, namely $[r_0-T_0,r_0+T_0]$, thanks to some cutoff function. Note first from \eqref{initial data soliton} that $\kappa^*(\dz,\nz,y)$ has no singularity for $y\in (-1,1)$, if $s_0$ is large enough. Accordingly, $u(r,0)$ has no singularity for $r\in [r_0-T_0,r_0+T_0]$. Therefore, by continuity, we may consider $\varepsilon_0>0$ small enough such that on $[r_0-T_0-2\varepsilon_0,r_0+T_0+2\varepsilon_0]$, $\uu(\dz,\nz,r_0,T_0,r,0)$ has no singularity (see \eqref{def uu}). Let us then consider the cut-off function $\chi\in \mathcal{C}^\infty_0$, such that $\chi\equiv 1$ on $[r_0-T_0-\varepsilon_0,r_0+T_0+\varepsilon_0]$, $\chi_0\equiv 0$ on $\R\backslash[r_0-T_0-2\varepsilon_0,r_0+T_0+2\varepsilon_0]$. Multiplying \eqref{singular initial data} with $\chi$ and denoting the result by $(u_0,u_1)$, we obtain \eqref{initial data u}. Clearly, $(U_0,U_1)\in H^1_{loc}\times L^2_{loc}(\R^N)$, where $U_i(x)=u_i(|x|)$.
\medskip

From the Cauchy theory applied to \eqref{eq initial}, there is a unique solution to the equation \eqref{eq radial} which is radial. From the uniqueness and finite time speed of propagation, the value of that solution in $\Gamma(r_0)$ defined in \eqref{Gamma r_0} is already known. It is given in \eqref{def u similarty var}. For simplicity, we will use the same notation for the solution, we have just got from the application of the Cauchy problem.

Note that, we have two cases for the domain $D$ of definition of $u$:\\
\textbf{Case 1:} $D=\{t\geq 0\}$.\\
\textbf{Case 2:} $D$ is included in  $\{0\leq t <T(r)\}$ for some 1-Lipschitz curve $r\mapsto T(r)$.

Let us show that Case 2 occurs with $T(r_0)=T_0$. We assume, by contradiction, that $u$ is global (Case 1) or that $u$ blows up with $T(r_0)>T_0$. We already know that $u$ is defined in $\Gamma(r_0)$. Hence, $T(r_0)\geq T_0$. In both cases, $u$ is defined on the cone $\mathcal{C}_{r_0,T_0+\eta_0,1}\cap \{t\geq 0\}$. In particular, the $L^2$ average of $u$ on slices of $\Gamma(r_0)$ should be bounded, at one hand .On the other hand, we have that

\begin{equation}\label{conclusion thm}
\begin{split}
\frac{1}{T_0-t}\int_{|r-r_0|< T_0-t}u^2dr&\geq \frac{1}{T_0-t}\int_{|r-r_0|<\frac{T_0-t}{2}}u^2dr\\
&=(T_0-t)^{-\frac{4}{p-1}}\int^{1/2}_{-1/2}w^2 dy\\
&\geq(T_0-t)^{-\frac{4}{p-1}}\int^{1/2}_{-1/2}w^2 \rho(y)dy.
\end{split}
\end{equation}
%$$\int_{|r-r_0|< T_0-t}u^2dr\geq\int_{|r-r_0|<\frac{T_0-t}{2}}(\partial_r u)^2dr\geq(T_0-t)^{-\frac{4}{p-1}}\int^{1/2}_{-1/2}(\partial_y w)^2 dy\geq(T_0-t)^{-\frac{4}{p-1}}\int^{1/2}_{-1/2}(\partial_y w)^2 \rho(y)dy$$
From Proposition \ref{sol conv solition}, we have that 
$$\int^{1/2}_{-1/2}w^2 \rho(y)dy\underset{t\rightarrow T_0}{\longrightarrow} \int^{1/2}_{-1/2}\kappa(\dhz)^2\rho(y) dy>0.$$
Together with \eqref{conclusion thm}, we conclude that 
$$\frac{1}{T_0-t}\int_{|r-r_0|< T_0-t}u^2dr\longrightarrow +\infty \text{ as } t\rightarrow + \infty,$$
and a contradiction follows.

Therefore, we have that case 2 occurs, and $u$ is defined below the graph of some 1-Lipschitz function $r\mapsto T(r)$ with $T(r_0)=T_0$. Since we have a convergence of $(w(s),\partial_s w(s))$ to one soliton in $\mathcal{H}$ in Proposition \ref{sol conv solition}, one may see from Theorem 6 in \cite{MZ2011BSM} that $r_0$ is a non-characteristic point. Since \eqref{conv in H} follows from Proposition \ref{sol conv solition}, we concludes the proof of Theorem \ref{Thm}.

\appendix

\section{Toolbox}
In this section, we give some useful tools required in this work. We will not give the proof since they are easy or already given in the cited references. We have first the following Hardy-Sobolev inequalities from \cite{MZ2012}.
\begin{lemma}[Hardy-Sobolev inequality]\label{estimates in different norms}

For $h\in \mathcal{H}_0$, we have
\begin{equation}\label{hardy-sobolev identity}
        ||h||_{L^{2}_{\frac{\rho}{1-y^2}}(-1,1)}+||h||_{L^{p+1}_\rho(-1,1)}+||h(1-y^2)^{\frac{1}{p-1}}||_{L^{\infty}_\rho(-1,1)}\leq C||h||_{\mathcal{H}_0},
\end{equation}
where $\rho$ and $\mathcal{H}_0$ are given in \eqref{def L} and \eqref{def H0}.
\end{lemma}
\begin{proof}
    See proof of Claim A.1 in \cite{MZ2012}
\end{proof}

We also have the following estimates on the solitons from \cite{MZ2012}:
\begin{lemma}[estimates on solitons]$ $
\begin{enumerate}[label=(\roman*)]
    \item For all $d\in (-1,1)$, it holds that
\begin{equation}\label{bound kappa H_0}
    ||\kappa(d)||_{\mathcal{H}_0}\leq CE(\kappa_0),
    \end{equation}
where
$$E(r(s)) = \int_{-1}^{1} \left( \frac{1}{2} (\partial_s r)^2 + \frac{1}{2} (\partial_y r)^2 (1 - y^2) + \frac{(p + 1)}{(p - 1)^2} r^2 - \frac{1}{p + 1} |r|^{p + 1} \right) \rho \, dy.
$$

\item For all $|d|<1$ and $\nu>-1+|d|$, then we have
\begin{equation}\label{kappa* norm lambda}
\begin{split}
||\kappa^*(d,\nu)||_{\mathcal{H}}\leq C\lambda(d,\nu)+C\mathds{1}_{\{\nu<0\}}\frac{|\nu|}{\sqrt{1-d^{2}}}\lambda^{\frac{p+1}{2}}(d,\nu),
\end{split}
\end{equation}
where $\lambda=\lambda(d,\nu)=\frac{(1-d^{2})^\frac{1}{p-1}}{((1+\nu)^2-d^{2})^\frac{1}{p-1}}$.
\item For  $|d|<1$ and $\frac{|\nu|}{1-|d|}\leq \varepsilon$, for some $\varepsilon>0$ small enough, we have
\begin{equation}\label{bound lambda-1}
\begin{split}
&|\lambda(d^*,\nu)^{1-p}-1| =\left|\left(1+\frac{\nu}{1-|d^*|}\right)\left(1+\frac{\nu}{1+|d^*|}\right)-1\right|\leq C\frac{|\nu|}{1-|d^*|}.
\end{split}    
\end{equation}
\end{enumerate}
\end{lemma}

\begin{proof}
    See \cite[Claim A.1, (A.1) and (i) of Claim A.2]{MZ2012} for \eqref{hardy-sobolev identity}, \eqref{bound kappa H_0}, \eqref{kappa* norm lambda}. Estimate \eqref{bound lambda-1} is easy to obtain.
\end{proof}
In following lemma,we give a bound on $\varphi$ using the norm of $q$.
\begin{lemma}[Estimate on $\varphi$]
We have the following estimate on  $\varphi$, defined in \eqref{def varphi}, given by the following
    \begin{equation}\label{ineq norm q2 & A}
    \frac{1}{C^*}||q||^2_{\mathcal{H}}\leq \varphi(q,q) \leq C^*||q||^2_{\mathcal{H}}.
    \end{equation}

\end{lemma}
\begin{proof}
The reader may see that the proof in \cite[Lemma C.1, (ii)]{MZ2012} combined with \eqref{consequances modulation technique} and \eqref{bound lambda-1} for $s_0$ large enough holds as well for our case.
\end{proof}

\textbf{Statements:} This work did not generate any datasets. All relevant information is included in the manuscript.

\printbibliography
\end{document}